\documentclass[a4paper,12pt]{amsart}
\usepackage[hmarginratio={1:1},vmarginratio={1:1},heightrounded,textwidth=420pt]{geometry}

\usepackage{amsfonts,amsmath,amssymb,amsthm}
\usepackage{mathtools}
\usepackage{graphicx}
\usepackage{hyperref}
\usepackage[utf8]{inputenc}
\usepackage{url}
\usepackage{floatrow}
\usepackage{tabularx}
\usepackage{subcaption}
\newcolumntype{Y}{>{\centering\arraybackslash}X}
\let\leq\leqslant
\let\geq\geqslant
\let\setminus\smallsetminus
\let\epsi\varepsilon
\let\rho\varrho

\newcommand{\mcJ}{\mathcal{J}}
\newcommand{\mcK}{\mathcal{K}}
\newcommand{\brac}[1]{{\left(#1\right)}}
\newcommand{\sbrac}[1]{{\left[#1\right]}}
\newcommand{\set}[1]{\left\{#1\right\}}
\newcommand{\norm}[1]{{\left|#1\right|}}
\newcommand{\floor}[1]{{\left\lfloor #1 \right\rfloor}}
\newcommand{\ceil}[1]{{\left\lceil #1 \right\rceil}}

\newcommand{\Oh}[1]{O\brac{#1}}
\newcommand{\oh}[1]{o\brac{#1}}

\makeatletter
\newcommand\ie{i.e\@ifnextchar.{}{.\@}}
\newcommand\etc{etc\@ifnextchar.{}{.\@}}
\newcommand\etal{et~al\@ifnextchar.{}{.\@}}
\makeatother

\newtheorem{theorem}{Theorem}

\begin{document}

\thispagestyle{empty}

\title{Lower Bounds for On-line Interval Coloring with Vector and Cardinality Constraints}

\author[G.~Gutowski]{Grzegorz Gutowski}
\author[P.~Mikos]{Patryk Mikos}

\address{
Theoretical Computer Science Department,%
\\Faculty of Mathematics and Computer Science,%
\\Jagiellonian University, Krak\'ow, Poland%
}
\email{\{gutowski,mikos\}@tcs.uj.edu.pl}

\begin{abstract}

We propose two strategies for Presenter in the on-line interval graph coloring games.
Specifically, we consider a setting in which each interval is associated with a $d$-dimensional vector of weights
  and the coloring needs to satisfy the $d$-dimensional bandwidth constraint, and the $k$-cardinality constraint.
Such a variant was first introduced by Epstein and Levy and it is a natural model for resource-aware task scheduling with $d$ different shared resources where at most $k$ tasks can be scheduled simultaneously on a single machine.

The first strategy forces any on-line interval coloring algorithm to use at least $\brac{5m-3}\frac{d}{\log d + 3}$ different colors on an $m\brac{\frac{d}{k} + \log{d} + 3}$-colorable set of intervals.
The second strategy forces any on-line interval coloring algorithm to use at least $\floor{\frac{5m}{2}}\frac{d}{\log d + 3}$ different colors on an $m\brac{\frac{d}{k} + \log{d} + 3}$-colorable set of unit intervals.


\end{abstract}

\maketitle

\section{Introduction}

A \emph{proper coloring} of a graph $G$ is an assignment of colors to the vertices of the graph such that adjacent vertices receive distinct colors.
A \emph{$k$-bounded coloring} of $G$ is a proper coloring of $G$ such that the number of vertices that receive any single color is at most $k$.
For a graph $G$, let $\chi(G)$ denote the \emph{chromatic number} of $G$, that is the minimum number of colors in a proper coloring of $G$, and let $\omega(G)$ denote the \emph{clique number} of $G$, that is, the maximum size of a subset of vertices such that any two vertices in the subset are adjacent.

An \emph{on-line graph coloring game} is a two-person game, played by Presenter and Algorithm.
In each round Presenter introduces a new vertex of a graph with its adjacency status to all vertices presented earlier.
Algorithm assigns a color to the incoming vertex in such a way that the coloring of the presented graph is proper.
The color of the new vertex is assigned before Presenter introduces the next vertex and the assignment is irrevocable.
The goal of Algorithm is to minimize the number of different colors used during the game.
In the \emph{$k$-bounded} variant of the game, the coloring constructed by Algorithm needs to be a $k$-bounded coloring of the presented graph.

For an interval $I = \sbrac{l,r}$ on the real line, we say that $l$ is the left endpoint, and $r$ is the right endpoint of $I$.
The \emph{length} of interval $I$ is the difference between its right endpoint and its left endpoint.
A set of intervals on the real line represents a graph in the following way.
Each interval represents a vertex and any two vertices are joined by an edge whenever the corresponding intervals intersect.
A graph which admits such a representation is an \emph{interval graph}.

An \emph{on-line interval coloring game} is a two-person game, again played by Presenter and Algorithm.
In each round Presenter introduces a new interval on the real line.
Algorithm assigns a color to the incoming interval in such a way that the coloring of the presented interval graph is proper, \ie{} all intervals of the same color are pairwise disjoint.
The color of the new interval is assigned before Presenter introduces the next interval and the assignment is irrevocable.
The goal of Algorithm is to minimize the number of different colors used during the game.

We consider a few variants of the on-line interval coloring game.
In the \emph{unit} variant of the game, all intervals presented by Presenter are of length exactly 1.
In the \emph{$d$-dimensional} variant of the game, Presenter associates a $d$-dimensional vector of weights from $\sbrac{0,1}$ with each presented interval.
Moreover, the coloring constructed by Algorithm needs to satisfy a different condition.
The condition is that for each color $\gamma$ and any point $p$ on the real line, the sum of weights of intervals containing $p$ and colored $\gamma$ does not exceed 1 in any of the coordinates.
In the \emph{$k$-cardinality} variant of the game, the coloring constructed by Algorithm needs to satisfy that for each color $\gamma$ and any point $p$ on the real line, the number of intervals containing $p$ and colored $\gamma$ does not exceed $k$.

We are most interested in the \emph{on-line $(k,d)$ interval coloring}, a variant in which each interval has a $d$-dimensional vector of weights and the coloring must satisfy constraints of both $d$-dimensional and $k$-cardinality variant.
That is, for each color $\gamma$ and any point $p$, the number of intervals containing $p$ and colored $\gamma$ does not exceed $k$, and the sum of weights of those intervals does not exceed 1 in any coordinate.

In the context of various on-line coloring games, the measure of the quality of a strategy for Algorithm is given by competitive analysis.
A coloring strategy for Algorithm is \emph{$r$-competitive} if it uses at most $r \cdot c$ colors for any $c$-colorable graph, or set of intervals, presented by Presenter.
The \emph{competitive ratio} for a problem is the infimum of all values $r$ such that there exists an $r$-competitive strategy for Algorithm for this problem.
In this paper we give lower bounds on competitive ratios for different problems.
We obtain these results by presenting explicit strategies for Presenter that force any Algorithm strategy to use many colors while the presented graph, or set of intervals, is colorable with a smaller number of colors.


We say that a strategy for Presenter in an on-line coloring problem is \emph{transparent} if after each time Algorithm assigns a color to a vertex, or interval, Presenter colors the vertex with his own color and reveals that color to Algorithm.
The coloring constructed by Presenter must satisfy the same conditions as the coloring constructed by Algorithm.
The number of colors used by a transparent strategy for Presenter gives an upper bound on the minimum number of colors that can be used in a coloring.

\subsection{Previous work}

There is a simple strategy for Presenter in on-line graph coloring game that forces Algorithm to use any number of colors while the constructed graph is $2$-colorable.
Thus, the competitive ratio for this problem is unbounded.
Nevertheless, it is an interesting question what is the competitive ratio when the on-line game is played only for at most $n$ rounds.
Halld\'{o}rsson and Szegedy~\cite{HalldorssonS94} presented a transparent strategy for Presenter that forces Algorithm to use at least $2\frac{n}{\log{n}}\brac{1+\oh{1}}$ different colors in $n$ rounds of the game while the constructed graph is $\log{n}\brac{1+\oh{1}}$-colorable.
The best known upper bound of $\Oh{\frac{n}{\log^{*}{n}}}$ on the competitive ratio for the $n$-round on-line graph coloring problem was given by Lovasz, Saks and Trotter~\cite{LovaszST89}.

The competitive ratio for the on-line interval coloring problem was established by Kierstead and Trotter~\cite{KiersteadT81}.
They constructed a strategy for Algorithm that uses at most $3\omega - 2$ colors while the clique size of the constructed graph is $\omega$.
They also presented a matching lower bound -- a strategy for Presenter that forces Algorithm to use at least $3\omega - 2$ colors.
Unit variant of the on-line interval coloring problem was studied by Epstein and Levy~\cite{EpsteinL05_2}.
They presented a strategy for Presenter that forces Algorithm to use at least $\floor{\frac{3\omega}{2}}$ colors while the clique size of the constructed graph is $\omega$.
Moreover, they showed that First-Fit algorithm uses at most $2\omega-1$ colors.
Epstein and Levy~\cite{EpsteinL05} introduced many variants of the on-line interval coloring problem.
The best known lower bound on the competitive ratio for the on-line $(k,d)$ interval coloring is $3$ for small $k$ and $\frac{24}{7}$ for large $k$.
For unit variant of this problem the best known lower bound is $\frac{3}{2}$.

Halld\'{o}rsson and Szegedy ideas were adopted by Azar~\etal{}~\cite{AzarCIKS13} to show lower bounds on competitive ratio for \emph{on-line $d$-vector bin packing}.
This problem is equivalent to a variant of $d$-dimensional on-line interval coloring where all presented intervals are the interval $\sbrac{0,1}$ with different vectors of weights.
Their strategy for Presenter shows that the competitive ratio for the on-line $d$-dimensional unit interval coloring problem is at least $2\frac{d}{\log^2d}\brac{1+\oh{1}}$.

\subsection{Our results}

We generalize Halld\'{o}rsson and Szegedy~\cite{HalldorssonS94} strategy into the setting of the $k$-bounded coloring, and using the technique similar to the one by Azar~\etal{}~\cite{AzarCIKS13} we adopt it to the on-line $(k,d)$ interval coloring problem.
We present how to combine this technique with classical results by Kierstead and Trotter~\cite{KiersteadT81}, and by Epstein and Levy~\cite{EpsteinL05_2,EpsteinL05}
to obtain a new lower bound of $5\frac{d}{\log{d}\brac{\frac{d}{k}+\log{d}}}\brac{1+\oh{1}}$ on the competitive ratio for the on-line $(k,d)$ interval coloring, and a lower bound of $\frac{5}{2}\frac{d}{\log{d}\brac{\frac{d}{k}+\log{d}}}\brac{1+\oh{1}}$ for unit variant of this problem.

\section{Graph coloring}

\begin{theorem}\label{thm:general_graphs_k_cardinality}
  For every $n \geq 2$ and $k \in \mathbb{N}_{+}$, there is a transparent strategy for Presenter that forces Algorithm to use at least $2\frac{n}{\log{n}+3}$ different colors in the $n$-round, $k$-bounded on-line graph coloring game and uses $\frac{n}{k} + \log{n} + 3$ colors.
\end{theorem}

\begin{proof}
Let $b = \floor{\log n} + 3$.
The state of a $k$-bounded on-line graph coloring game is represented by a \emph{progress matrix} $M$.
Each cell $M[i,j]$ is either empty or it contains exactly one vertex.
At the beginning of the game, all cells are empty.
A vertex in $M[i,j]$ is colored by Algorithm with color $j$ and by Presenter with color $i$.
Each player can use a single color $\gamma$ to color at most $k$ vertices, so there are at most $k$ vertices in any column, and in any row of the progress matrix.
We say that a row with $k$ vertices is \emph{depleted}.
Presenter can no longer use colors corresponding to depleted rows.
Presenter maintains a set of exactly $b$ \emph{active} rows, denoted $\mathcal{A}$, that contains all nonempty non-depleted rows and additionally some empty rows ($\set{i: 1 \leq \norm{\cup_{j}M[i,j]} < k} \subset \mathcal{A}$ and $\norm{\mathcal{A}} = b$).
At the beginning of the game there are no depleted rows and $\mathcal{A} = \set{1,\ldots,b}$.
When some row becomes depleted then it is removed from $\mathcal{A}$ and a new empty row is added to $\mathcal{A}$.
A \emph{pattern} is a subset of rows.
We say that a pattern $p$ \emph{represents} a column $j$ if $\forall{i}: i \in p \iff M[i,j] \neq \emptyset$. 
A pattern $p$ is \emph{active} if it is a nonempty subset of $\mathcal{A}$ such that $\norm{p} \leq \floor{\frac{b}{2}}$.
An active pattern $p$ is \emph{present} in $M$ if at least one column of $M$ is represented by $p$.

\begin{table}[H]
  \caption{Example of a progress matrix after 15 rounds}
  \label{fig:Mdk}
  \begin{tabularx}{0.7\textwidth}{*{12}{|Y}}\hline

             & 1 & 2 & 3 & 4 & 5 & 6 & 7 & 8 & 9 & 10 & $\ldots$ \\ \hline

    $1$      & $v_{1}$  &          & $v_{3}$  & $v_{5}$  & $v_{7}$  &          &          &          &          &  & \\ \hline
    $2^{*}$  &          & $v_{2}$  &          &          &          & $v_{10}$ &          & $v_{12}$ &          &  & \\ \hline
    $3$      &          & $v_{15}$ & $v_{4}$  &          &          &          & $v_{9}$  & $v_{13}$ &          &  & \\ \hline
    $4^{*}$  & $v_{6}$  &          &          &          &          &          & $v_{11}$ &          &          &  & \\ \hline
    $5^{*}$  &          &          &          &          &          & $v_{8}$  &          &          & $v_{14}$ &  & \\ \hline
    $6^{*}$  &          &          &          &          &          &          &          &          &          &  & \\ \hline
    $7$      &          &          &          &          &          &          &          &          &          &  & \\ \hline
    $\ldots$ &          &          &          &          &          &          &          &          &          &  &
  \end{tabularx}
\end{table}

\autoref{fig:Mdk} shows a possible state of the progress matrix $M$ after 15 rounds of the 4-bounded on-line graph coloring game with $n=4$ and $b=4$.
In this example, the last introduced vertex $v_{15}$ is colored by Algorithm with color 2 and by Presenter with color 3.
Rows 1 and 3 are depleted and the set of active rows is $\mathcal{A} = \set{2,4,5,6}$.
There are 10 different active patterns, but only 2 of them are present in $M$: pattern $\set{2,5}$ in column 6, and pattern $\set{5}$ in column 9.

The transparent strategy for Presenter for round $t$ is as follows:
\begin{enumerate}
  \item\label{rul:patt} Choose an active pattern $p_t$ that is not present in $M$.
  \item\label{rul:neib} Introduce a new vertex $v_t$ that is adjacent to all vertices colored by Presenter with colors not in $p_t$.
  \item Algorithm colors $v_t$ with color $\gamma$.
  \item\label{rul:resp} Color $v_t$ with any color $\rho$ such that $\rho \in p_t$ and $M[\rho,\gamma] = \emptyset$.
\end{enumerate}

We claim that Presenter can follow this strategy as long as there is an active pattern not present in the progress matrix.
To prove that, we need to show that in step~\eqref{rul:resp} Presenter always can choose an appropriate color $\rho$, and that the coloring constructed by Presenter is a $k$-bounded coloring of the constructed graph.

Let $q$ be a pattern that represents column $M[*,\gamma]$ of the progress matrix before round $t$.
We claim that $q \subsetneq p_t$.
Assume to the contrary that $i \in q \setminus p_t$.
It follows, that there is a vertex $v$ in cell $M[i,\gamma]$ and by rule~\eqref{rul:neib} $v$ is adjacent to $v_t$.
Thus, Algorithm cannot color vertex $v_t$ with color $\gamma$.
Pattern $q$ is present in $M$ before round $t$ and by rule~\eqref{rul:patt} pattern $p_t$ is not present in $M$ before round $t$.
It follows that $q$ is a strict subset of $p_t$ and Presenter has at least one choice for color $\rho$ in step~\eqref{rul:resp}.

When Presenter assigns color $\rho$ to vertex $v_t$, we have that $\rho \in p_t$; $p_t$ is an active pattern; $\rho$ is an active row, and there are less than $k$ vertices colored by Presenter with $\rho$.
Rule~\eqref{rul:neib} asserts that none of the vertices adjacent to $v_t$ is colored with any of the colors in $p_t$.
Thus, we have that Presenter can follow the strategy as long as there is a choice of an appropriate pattern in step~\eqref{rul:patt}. 

We claim that the game can be played for at least $n$ rounds.
Indeed, there are $\binom{b}{x}$ different patterns of size $x$ and each one of them must represent a column of the progress matrix with exactly $x$ vertices.
Thus, when all active patterns represent some column of the progress matrix, the game has been played for at least $\sum_{1 \leq x \leq \floor{\frac{b}{2}}} x\binom{b}{x} \geq n$ rounds.

After $n$ rounds, Presenter used colors corresponding to depleted and active rows.
There are at most $\floor{\frac{n}{k}}$ depleted rows and exactly $\floor{\log n} + 3$ active rows.
Thus, Presenter uses at most $\frac{n}{k} + \log n + 3$ colors in the first $n$ rounds.

Let $q_j$ be a pattern representing column $M[*,j]$ after $n$ rounds.
Let $t$ be the last round when a vertex was added to column $j$.
We have that $q_j$ is a subset of pattern $p_t$ which was an active pattern before round $t$, and the size of $q_j$ is at most $\floor{\frac{b}{2}}$.
Thus, there are at least $2\frac{n}{\log{n}+3}$ nonempty columns after $n$ rounds.

\end{proof}

For fixed parameters $n$ and $k$, denote a generalized Halld\'{o}rsson and Szegedy strategy by $HS_{k,n}$.
Note that for $k = +\infty$, there are no depleted rows in matrix $M$ and $k$-bounded coloring is simply a proper coloring.
In this case we get the original Halld\'{o}rsson and Szegedy result for the on-line graph coloring problem.

\begin{theorem}[Halld\'{o}rsson, Szegedy~\cite{HalldorssonS94}]\label{thm:general_graphs}
For every integer $n \geq 2$, there is a transparent strategy for Presenter that forces Algorithm to use at least $2\frac{n}{\log{n}+3}$ colors in the $n$-round on-line graph coloring game and uses $\log{n}+3$ colors.
\end{theorem}

\section{Interval coloring}

In the proof of the following theorem, we use strategy $HS_{k,d}$ to show a lower bound on the competitive ratio for the on-line $(k,d)$ interval coloring problem.

\begin{theorem}\label{thm:kd_lower_general}
For every $d \geq 2$ and $k,m \in \mathbb{N}_{+}$, there is a strategy for Presenter that forces Algorithm to use at least $\brac{5m-3}\frac{d}{\log{d}+3}$ different colors in the on-line $(k,d)$ interval coloring game while the constructed set of intervals is $m\brac{\frac{d}{k} + \log{d} + 3}$-colorable.
\end{theorem}

\begin{proof}
For any fixed parameters $k \in \mathbb{N}_{+}, d \geq 2$, $L < R$, $\epsi \in (0, \frac{1}{d})$ we describe an auxiliary strategy $HS_{k,d}(\epsi, L, R)$.
Let $\alpha = 1 - \frac{1}{2}\epsi$, $\delta = \frac{1}{2d}\epsi$.
In the $t$-th round of the on-line $(k,d)$ interval coloring game, Presenter uses $HS_{k,d}$ strategy to get a new vertex $v_t$.
Then, presents an interval $\sbrac{L,R}$ with weights $w_t$, where $w_t = \brac{x_1,\ldots,x_d}$ is a $d$-dimensional vector with $x_t=\alpha$, $x_i=\epsi$ for all $i < t$ such that $v_i$ is adjacent to $v_t$, and $x_i=\delta$ in every other coordinate.
Figure~\ref{fig:interval_encoding} shows an example of $w_{6}$ for a vertex $v_6$ that is adjacent to $v_2$ and $v_5$.

\begin{figure}[H] \setlength{\unitlength}{0.09in}
\centering
\begin{picture}(30,14)

\put(0,1.5){\line(1,0){30}}
\put(0,1.5){\line(0,1){1}}
\put(3,1.5){\line(0,1){3}}
\put(6,1.5){\line(0,1){3}}
\put(9,1.5){\line(0,1){1}}
\put(12,1.5){\line(0,1){3}}
\put(15,1.5){\line(0,1){12}}
\put(18,1.5){\line(0,1){12}}
\put(21,1.5){\line(0,1){1}}
\put(24,1.5){\line(0,1){1}}
\put(27,1.5){\line(0,1){1}}
\put(30,1.5){\line(0,1){1}}

\put(0,2.5){\line(1,0){3}}
\put(3,4.5){\line(1,0){3}}
\put(6,2.5){\line(1,0){3}}
\put(9,2.5){\line(1,0){3}}
\put(12,4.5){\line(1,0){3}}
\put(15,13.5){\line(1,0){3}}
\put(18,2.5){\line(1,0){3}}
\put(21,2.5){\line(1,0){3}}
\put(27,2.5){\line(1,0){3}}

\put(1,0){1}
\put(4,0){2}
\put(7,0){3}
\put(10,0){4}
\put(13,0){5}
\put(16,0){6}
\put(19,0){7}
\put(22,0){8}
\put(24.5,0){$\ldots$}
\put(28,0){$d$}

\put(1,3){$\delta$}
\put(4,5){$\epsi$}
\put(7,3){$\delta$}
\put(10,3){$\delta$}
\put(13,5){$\epsi$}
\put(16,14){$\alpha$}
\put(19,3){$\delta$}
\put(22,3){$\delta$}
\put(28,3){$\delta$}

\end{picture} 
\caption{Encoding of $v_6$ in a $d$-dimensional vector of weights}
\label{fig:interval_encoding}
\end{figure}
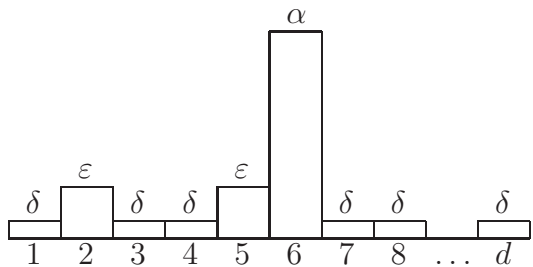

$\brac{\sbrac{L,R},w_t}$ is colored by Algorithm with color $\gamma_t$.
Then, $\gamma_t$ is forwarded to $HS_{k,d}$ as the color of $v_t$.
$HS_{k,d}$ colors $v_t$ with $\rho_t$, but Presenter discards that information.
See Figure~\ref{fig:hs_encoding_diagram} for a diagram of the strategy $HS_{k,d}(\epsi, L, R)$, and Figure~\ref{fig:hs_encoding_example} for an example encoding of a graph.

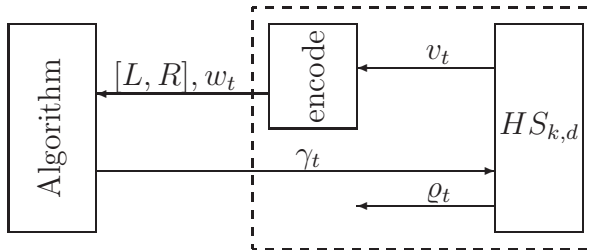
\begin{figure}[H] \setlength{\unitlength}{0.09in}
\centering
\begin{picture}(34,15) 

\put(0,2){\framebox(5,12){\rotatebox{90}{Algorithm}}}

\put(14,1){\dashbox{0.5}(20,14)}
\put(15,8){\framebox(5,6){\rotatebox{90}{encode}}}
\put(28,2){\framebox(5,12){$HS_{k,d}$}}

\put(24,12){$v_t$}
\put(28,11.5){\vector(-1,0){8}}

\put(6,10.5){$\sbrac{L,R}, w_t$}
\put(15,10){\vector(-1,0){10}}

\put(16.5,6){$\gamma_t$}
\put(5,5.5){\vector(1,0){23}}

\put(24,4){$\rho_t$}
\put(28,3.5){\vector(-1,0){8}}

\end{picture} 
\caption{Encoding of $HS_{k,d}$ strategy} 
\label{fig:hs_encoding_diagram}
\end{figure}

We claim that the encoding strategy ensures that any intervals $I_i$ and $I_j$ can get the same color iff vertices $v_i$ and $v_j$ are not adjacent.
First, assume that $i < j$ and $v_i$ is adjacent to $v_j$.
Vector $w_i$ has $\alpha$ in the $i$-th coordinate, vector $w_j$ has $\epsi$ in the $i$-th coordinate, and $\alpha + \epsi > 1$.
Thus, intervals $I_i$ and $I_j$ must be colored with different colors.
Let $\mcJ \subset \set{I_1,...,I_{t-1}}$ be the set of intervals colored with $\gamma$ before round $t \leq d$.
Assume that $v_t$ is not adjacent to any of the vertices in $\mcJ$.
Denote the $l$-th coordinate of the sum of vectors of weights of intervals in $\mcJ$ by $W_l$.
For any $1 \leq l \leq d$, if $I_l \in \mcJ$ then $W_l = \alpha + \delta\brac{\norm{\mcJ}-1} < 1 - \delta$.
In this case we have that $v_t$ is not adjacent to $v_l$, and that the $l$-th coordinate of the vector of weights of $v_t$ is $\delta$.
If $I_l \notin \mcJ$ then $W_l \leq \epsi\norm{\mcJ} < 1 - \epsi$.
For $l = t$, we have $W_l \leq \delta\norm{\mcJ} \leq 1 - \alpha$.
Thus, the sum of vector of weights of the intervals in the set $\mcJ \cup \set{I_t}$ does not exceed $1$ in any coordinate and $I_t$ can be colored with $\gamma$.

\begin{figure}[H] \setlength{\unitlength}{0.09in}
\centering
\begin{picture}(36,12) 
\put(0,4){\circle{0.8}}
\put(4,0){\circle{0.8}}
\put(8,4){\circle{0.8}}
\put(12,8){\circle{0.8}}
\put(16,0){\circle{0.8}}
\put(18,4){\circle{0.8}}

\put(4,0){\vector(-1,1){3.9}}
\put(8,4){\vector(-1,-1){3.9}}
\put(12,8){\vector(-3,-1){11.9}}
\put(12,8){\vector(-1,-1){3.9}}
\put(16,0){\vector(-1,0){11.9}}
\put(16,0){\vector(-2,1){7.9}}
\put(18,4){\vector(-1,-2){1.9}}
\put(18,4){\vector(-3,2){5.9}}

\put(0,5){$v_1$}
\put(3,1.2){$v_2$}
\put(7,5){$v_3$}
\put(12,9){$v_4$}
\put(14.5,1){$v_5$}
\put(18,5){$v_6$}
\put(25,5){
\begin{tabular}{| l || l | l | l | l | l | l |} \hline
$w_1$ & $\alpha$ & $\delta$ & $\delta$ & $\delta$ & $\delta$ & $\delta$ \\ \hline
$w_2$ & $\epsi$ & $\alpha$ & $\delta$ & $\delta$ & $\delta$ & $\delta$ \\ \hline
$w_3$ & $\delta$ & $\epsi$ & $\alpha$ & $\delta$ & $\delta$ & $\delta$ \\ \hline
$w_4$ & $\epsi$ & $\delta$ & $\epsi$ & $\alpha$ & $\delta$ & $\delta$ \\ \hline
$w_5$ & $\delta$ & $\epsi$ & $\epsi$ & $\delta$ & $\alpha$ & $\delta$ \\ \hline
$w_6$ & $\delta$ & $\delta$ & $\delta$ & $\epsi$ & $\epsi$ & $\alpha$ \\ \hline
\end{tabular}
}

\end{picture}

\caption{Example of a graph and vectors of weights corresponding to the vertices}
\label{fig:hs_encoding_example}
\end{figure}
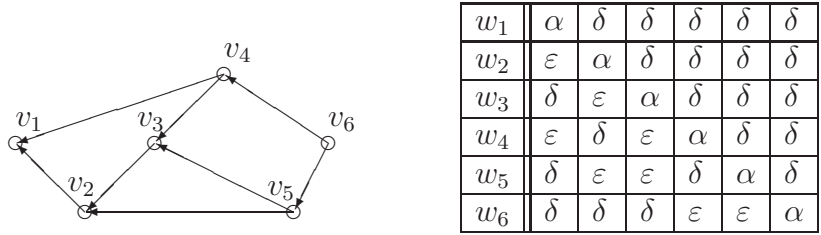

Consider a sequence of parameters $\set{\epsi_i}_{i \in \mathbb{N}_{+}}$ defined as $\epsi_i \coloneqq \left(\frac{1}{2d}\right)^{i}$.
See that for every $i \in \mathbb{N}_{+}$ we have $\epsi_i \in (0, \frac{1}{d})$ and we can use $HS_{k,d}(\epsi_i, L, R)$ strategy.
Let $\alpha_i = 1 - \frac{1}{2}\epsi_i$ and $\delta_i = \frac{1}{2d}\epsi_i = \epsi_{i+1}$.

Let $\mcJ_i$ be a set of intervals constructed by $HS_{k,d}(\epsi_i,L_i,R_i)$ strategy and $\mcJ_j$ be a set of intervals constructed by $HS_{k,d}(\epsi_j,L_j,R_j)$ strategy.
Assume that $i < j$, $[L_i, R_i] \cap [L_j, R_j] \neq \emptyset$ and that the construction of $\mcJ_i$ is finished before the construction of $\mcJ_j$ starts.
Any interval $I \in \mcJ_j$ has weight $\alpha_j$ in one of the coordinates and every interval in $\mcJ_i$ has weight either $\alpha_i, \epsi_i$ or $\delta_i$ in that coordinate.
In any case, sum of those weights exceeds $1$ and no two intervals, one in $\mcJ_i$, other in $\mcJ_j$ can be colored with the same color.

The rest of the proof uses a technique similar to the one by Kierstead and Trotter~\cite{KiersteadT81}.
For $m \in \mathbb{N}_{+}$, let $c_m = \brac{5m-3}\frac{d}{\log{d}+3}$, and $o_m = m\brac{\frac{d}{k} + \log{d} + 3}$.
By induction on $m$, we show a strategy $S_m$ for Presenter such that: all introduced intervals are contained in a fixed region $\sbrac{A,B}$; all intervals come from calls of strategies $HS_{k,d}(\epsi,L,R)$ with $\epsi$ in $\set{\epsi_1,\ldots,\epsi_{3m}}$; Algorithm uses at least $c_m$ different colors; constructed set of intervals is $o_m$-colorable.
For $m = 1$ and a fixed region $\sbrac{A,B}$, Presenter uses strategy $HS_{k,d}(\epsi_1,A,B)$.
This strategy forces Algorithm to use at least $c_1$ different colors, and the constructed set of intervals is $o_1$-colorable.
Thus, in this case we are done.

Let $\bar{c} = 3{c_{m+1} \choose c_{m}} + 1$.
Presenter splits the fixed region $\sbrac{A,B}$ into $\bar{c}$ disjoint regions $\sbrac{l_1,r_1},\ldots,\sbrac{l_{\bar{c}},r_{\bar{c}}}$.
By induction, in each region Presenter can use strategy $S_m$ independently.
As a result, in each region $\sbrac{l_i,r_i}$, we get a set of intervals $\mcJ_i$.
If during the construction Algorithm uses at least $c_{m+1}$ colors, we are done.
Otherwise, let $C_i$ be a $c_m$-element subset of colors used by Algorithm to color $\mcJ_i$.
Some $c_m$-element set of colors $C^*$ appears on the list $\brac{C_1,\ldots,C_{\bar{c}}}$ at least 4 times.
Let $a,b,c,d \in\set{1,\ldots,\bar{c}}$, $a < b < c < d$ be indices such that $C_a = C_b = C_c = C_d = C^*$. 
Define $p_i = \frac{1}{2}\brac{r_i + l_{i+1}}$ for $i=a,b,c$.

Presenter uses strategy $HS_{k,d}(\epsi_{3m+1},l_a,p_a)$ to get a set of intervals $\mcK_1$ and then strategy $HS_{k,d}(\epsi_{3m+1},\frac{r_c+p_c}{2},r_d)$ to get a set of intervals $\mcK_2$.
Let $D_1$ be a $c_1$-element subset of colors used by Algorithm to color $\mcK_1$, and $D_2$ be a $c_1$-element subset of colors used to color $\mcK_2$.
The construction of $\mcJ_a$ and $\mcJ_d$ is finished before the construction of $\mcK_1$ and $\mcK_2$ started and region $\sbrac{l_a,p_a}$ covers $\sbrac{l_a,r_a}$, and region $\sbrac{\frac{r_c+p_c}{2},r_d}$ covers $\sbrac{l_d,r_d}$.
Thus, none of the colors in $C^*$ can be used to color any interval in $\mcK_1$ or $\mcK_2$.
Now, the strategy splits into two cases: we either have $\norm{D_1 \cap D_2} \leq \frac{1}{2}c_1$ or $\norm{D_1 \cap D_2} > \frac{1}{2}c_1$.

\noindent \textbf{Case 1.} $\norm{D_1 \cap D_2} \leq \frac{1}{2}c_1$.
Presenter uses strategy $HS_{k,d}(\epsi_{3m+2},\frac{r_a+p_a}{2},p_c)$ to get a set of intervals $\mcK_3$.
Let $D_3$ be the set of colors used to color $\mcK_3$.
See Figure~\ref{fig:construct123} for a diagram of the construction.
Region $\sbrac{\frac{r_a+p_a}{2},p_c}$ covers $\sbrac{l_b,r_b}$ and we get $C^* \cap D_3 = \emptyset$.
Moreover, any interval in $\mcK_3$ intersects any interval in $\mcK_1$, and any interval in $\mcK_2$.
Thus, $D_3 \cap D_1 = \emptyset$, $D_3 \cap D_2 = \emptyset$, and Algorithm uses at least $\norm{C^* \cup D_1 \cup D_2 \cup D_3} \geq c_m + c_1 + \frac{1}{2}c_1 + c_1 = c_{m+1}$ colors.
Each set of intervals $\mcJ_1, \ldots, \mcJ_{\bar{c}}$ intersects with intervals in at most one of the sets $\mcK_1$, $\mcK_2$, or $\mcK_3$.
Thus, all presented intervals can be colored with $\max\{2,m+1\}(\frac{d}{k} + \log{d} +3) = o_{m+1}$ colors and in this case we are done.

\begin{figure}[H] \setlength{\unitlength}{0.09in}
\centering
\begin{picture}(36,10) 

\put(0,9) {$l_a$}
\put(4,9) {$r_a$}
\put(0,6){\framebox(5,2){$C_a$}}
\put(10,9) {$l_b$}
\put(14,9) {$r_b$}
\put(10,6){\framebox(5,2){$C_b$}} 
\put(20,9) {$l_c$}
\put(24,9) {$r_c$}
\put(20,6){\framebox(5,2){$C_c$}}
\put(30,9) {$l_d$}
\put(34,9) {$r_d$}
\put(30,6){\framebox(5,2){$C_d$}}

\put(7,9.5) {$p_a$}
\put(7,8.8){\vector(0,-1){3}}
\put(17,9.5) {$p_b$}
\put(17,8.8){\vector(0,-1){3}} 
\put(27,9.5) {$p_c$}
\put(27,8.8){\vector(0,-1){3}} 

\put(0,3){\framebox(7,2){$D_1$}}
\put(26,3){\framebox(9,2){$D_2$}}
\put(6,0){\framebox(21,2){$D_3$}}

\end{picture} 
\caption{Construction in case $\norm{D_1 \cap D_2} \leq \frac{1}{2}c_1$} 
\label{fig:construct123}
\end{figure}
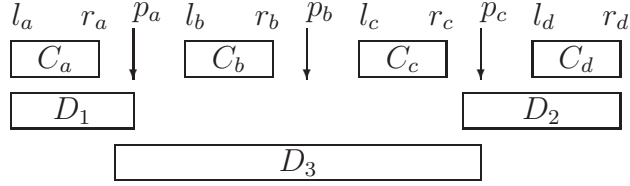

\noindent \textbf{Case 2.} $\norm{D_1 \cap D_2} > \frac{1}{2}c_1$.
Presenter uses strategy $HS_{k,d}(\epsi_{3m+2},\frac{r_b+p_b}{2},p_c)$ to get a set of intervals $\mcK_4$ and then strategy $HS_{k,d}(\epsi_{3m+3},\frac{r_a+p_a}{2},p_b)$ to get $\mcK_5$.
Let $D_4$ be a $c_1$-element subset of colors used by Algorithm to color $\mcK_4$, and $D_5$ be the set of colors used to color $\mcK_5$.
See Figure~\ref{fig:construct1245} for a diagram of the construction.
Similar argument as in the previous case gives $C^* \cap D_4 = \emptyset$, $C^* \cap D_5 = \emptyset$, $D_2 \cap D_4 = \emptyset$, $D_1 \cap D_5 = \emptyset$, and $D_4 \cap D_5 = \emptyset$.
Set $D_2$ contains at least $\frac{1}{2}c_1$ elements from the set $D_1$.
Thus, we have $\norm{D_4 \cap D_1} \leq \frac{1}{2}c_1$.
Algorithm uses at least $\norm{C^* \cup D_1 \cup D_4 \cup D_5} \geq c_m + c_1 + \frac{1}{2}c_1 + c_1 = c_{m+1}$ colors.
Each set of intervals $\mcJ_1, \ldots, \mcJ_{\bar{c}}$ intersects with intervals in at most one of the sets $\mcK_1$, $\mcK_2$, $\mcK_4$, or $\mcK_5$.
Thus, all intervals can be colored with $\max\{2,m+1\}(\frac{d}{k} + \log{d} +3) = o_{m+1}$ colors.

\begin{figure}[H] \setlength{\unitlength}{0.09in}
\centering
\begin{picture}(36,13) 

\put(0,12) {$l_a$}
\put(4,12) {$r_a$}
\put(0,9){\framebox(5,2){$C_a$}}
\put(10,12) {$l_b$}
\put(14,12) {$r_b$}
\put(10,9){\framebox(5,2){$C_b$}} 
\put(20,12) {$l_c$}
\put(24,12) {$r_c$}
\put(20,9){\framebox(5,2){$C_c$}}
\put(30,12) {$l_d$}
\put(34,12) {$r_d$}
\put(30,9){\framebox(5,2){$C_d$}}

\put(7,12.5) {$p_a$}
\put(7,11.8){\vector(0,-1){3}}
\put(17,12.5) {$p_b$}
\put(17,11.8){\vector(0,-1){3}} 
\put(27,12.5) {$p_c$}
\put(27,11.8){\vector(0,-1){3}} 

\put(0,6){\framebox(7,2){$D_1$}}
\put(26,6){\framebox(9,2){$D_2$}}
\put(16,3){\framebox(11,2){$D_4$}}
\put(6,0){\framebox(11,2){$D_5$}}

\end{picture} 
\caption{Construction in case $\norm{D_1 \cap D_2} \ge \frac{1}{2}c_1$} 
\label{fig:construct1245}
\end{figure}

\end{proof}


\section{Unit interval coloring}

\begin{theorem}\label{thm:kd_lower_unit}
For every $d \geq 2$ and $k,m \in \mathbb{N}_{+}$, there is a strategy for Presenter that forces Algorithm to use at least $\floor{\frac{5m}{2}}\frac{d}{\log{d}+3}$ different colors in the on-line $(k,d)$ unit interval coloring game while the constructed set of intervals is $m\brac{\frac{d}{k} + \log{d} + 3}$-colorable.
\end{theorem}

\begin{proof}

The proof combines strategy $HS_{k,d}(\epsi,L,R)$ introduced in the proof of Theorem~\ref{thm:kd_lower_general} with technique similar to the one by Epstein and Levy~\cite{EpsteinL05_2,EpsteinL05}.
Assume that the sequence of encoding parameters $\{\epsi_i\}_{i \in \mathbb{N}_{+}}$ is defined the same way as in the proof of Theorem~\ref{thm:kd_lower_general}.

The strategy consists of 3 phases.
In the \emph{initial phase} Presenter uses strategy $HS_{k,d}(\epsi_i,0,1)$ for $i=1,\ldots,\floor{\frac{m}{2}}$ sequentially.
There is a coloring of all intervals introduced in the initial phase using $\floor{\frac{m}{2}}\brac{\frac{d}{k} + \log{d} + 3}$ colors, but Algorithm uses at least $\floor{\frac{m}{2}}\frac{2d}{\log{d}+3}$ colors.
Let $C_{init}$ be a $\floor{\frac{m}{2}}\frac{2d}{\log{d}+3}$-element subset of colors used by Algorithm in the initial phase.

For $L < R < L+1$, let $Sep(L,R)$ be the \emph{separation strategy} that introduces $d$ unit intervals in the following way.
Initialize $l = L$, and $r = R$.
To get next interval, calculate $p = \frac{1}{2}(l+r)$ and introduce interval $I = \sbrac{p, p+1}$.
If Algorithm colors $I$ with color in $C_{init}$, then update $r = p$.
Otherwise, \emph{mark} interval $I$ and update $l = p$.
Observe that to the left of $p$ there are only left endpoints of marked intervals.
Moreover, all introduced intervals have nonempty intersection.

The \emph{separation phase} consists of $2\floor{\frac{m}{2}}$ subphases.
We fix $L_1=\frac{3}{2}$ and $R_1=2$.
For $i=2,\ldots,2\floor{\frac{m}{2}}$, points $L_i$ and $R_i$ are established after the $\brac{i-1}$-th subphase.
Denote by $Sub_{i}$, the strategy for the $i$-th subphase being a combination of the $HS_{k,d}(\epsi_{i},0,1)$ strategy and the $Sep(L_i,R_i)$ strategy.
Strategy $Sub_{i}$ introduces $d$ intervals.
The position of each interval is determined using $Sep(L_i,R_i)$ strategy, and the $d$-dimensional vector of weights associated with each interval is determined according to $HS_{k,d}(\epsi_{i},0,1)$.
See Figure~\ref{fig:strategy_sub} for a diagram of the strategy $Sub_i$.

At the end of each subphase, Presenter decides whether the subphase is \emph{marked} or not.
The set of marked subphases is denoted by $\mathcal{M}$.
Let $C_i$ be the set of colors used by Algorithm in the $i$-th subphase and not present in the set $C_{init}$.
Subphase $i$ is marked if and only if one of the following conditions holds:
\begin{enumerate}
\item the number of remaining subphases including the $i$-th is $\floor{\frac{m}{2}} - \norm{\mathcal{M}}$.
\item $\norm{C_{i}} \geq \frac{d}{\log{d}+3}$ and $\norm{\mathcal{M}} < \floor{\frac{m}{2}}$,
\end{enumerate}
Observe that at the end of the separation phase we have exactly $\floor{\frac{m}{2}}$ marked subphases.

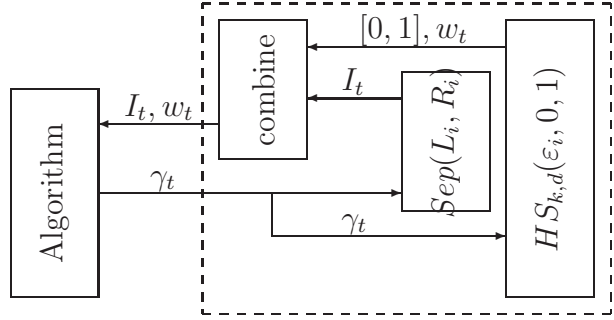
\begin{figure}[H] \setlength{\unitlength}{0.09in}
\centering
\begin{picture}(35,20) 

\put(0,1){\framebox(5,12){\rotatebox{90}{Algorithm}}}
\put(12,9){\framebox(5,8){\rotatebox{90}{combine}}}
\put(11,0){\dashbox{0.5}(23.5,18)}

\put(22.5,6){\framebox(5,8){\rotatebox{90}{$Sep(L_i,R_i)$}}}

\put(28.5,1){\framebox(5,16){\rotatebox{90}{$HS_{k,d}(\epsi_i,0,1)$}}}

\put(20,16){$[0,1],w_t$}
\put(28.5,15.5){\vector(-1,0){11.5}}

\put(19,13){$I_t$}
\put(22.5,12.5){\vector(-1,0){5.5}}

\put(6.5,11.5){$I_t, w_t$}
\put(12,11){\vector(-1,0){7}}

\put(8,7.5){$\gamma_t$}
\put(5,7){\vector(1,0){17.5}}

\put(19,5){$\gamma_t$}
\put(15,7){\line(0,-1){2.5}}
\put(15,4.5){\vector(1,0){13.5}}


\end{picture} 
\caption{Strategy $Sub_{i}$ for a single subphase} 
\label{fig:strategy_sub}
\end{figure}

Let $L^{*}$ be the left endpoint of the leftmost interval introduced in the $i$-th subphase.
Let $L$ be the left endpoint of the rightmost interval introduced in the $i$-th subphase and colored by Algorithm with a color $c \notin C_{init}$.
Set $L = L_i$ if such an interval does not exist.
Let $R$ be the left endpoint of the leftmost interval introduced in the $i$-th subphase and colored by Algorithm with a color $c \in C_{init}$.
Set $R = R_i$ if such an interval does not exist.
If subphase $i$ is marked then $L_{i+1} = L$ and $R_{i+1} = R$.
Otherwise, $L_{i+1} = L_i$ and $R_{i+1} = L^{*}$.
This completes the definition of the separation phase.

%
%
%
%
%
%
%
%
%
%

Let $m' = 2\floor{\frac{m}{2}}$ and $P = \frac{1}{2}(L_{m'+1} + R_{m'+1})$.
Observe that every interval introduced in the separation phase with the left endpoint to the left of $P$ belongs to a marked subphase and is colored with a color $c \notin C_{init}$.
Let $C_{sep}$ be the set of colors used in the separation phase to color intervals with the left endpoint to the left of $P$.

In each subphase Algorithm uses at least $\frac{2d}{\log{d}+3}$ different colors, so in the separation phase Algorithm uses at least $2\floor{\frac{m}{2}}\frac{2d}{\log{d}+3}$ colors in total.
Because $\norm{C_{init}} = \floor{\frac{m}{2}}\frac{2d}{\log{d}+3}$, Algorithm, in the separation phase, uses at least $\floor{\frac{m}{2}}\frac{2d}{\log{d}+3}$ colors not in $C_{init}$.
The set of marked subphases $\mathcal{M}$ contains $x$ subphases in which Algorithm used at least $\frac{d}{\log{d}+3}$ such colors and $\floor{\frac{m}{2}} - x$ last subphases.
From the first $\floor{\frac{m}{2}} + x$ subphases only $x$ subphases are marked.
By the definition, in an unmarked subphase $i$ for $i \leq \floor{\frac{m}{2}} + x$, Algorithm uses less than $\frac{d}{\log{d}+3}$ colors not in $C_{init}$.
Thus, at most $\floor{\frac{m}{2}}\frac{d}{\log{d}+3}$ such colors from subphases $1$ up to $\floor{\frac{m}{2}} + x$ are not in the set $C_{sep}$.
All colors not in $C_{init}$ used in the subphase $i$ for $i > \floor{\frac{m}{2}} + x$ are in the set $C_{sep}$.
Thus, $\norm{C_{sep}} \geq \floor{\frac{m}{2}}\frac{d}{\log{d}+3}$.

In the \emph{final phase}, Presenter uses strategy $HS_{k,d}(\epsi_i,P-1, P)$ for $i=m+1,\ldots,m+\ceil{\frac{m}{2}}$ sequentially.
Every interval introduced in the final phase intersects with every interval from the initial phase and every interval from the separation phase with the left endpoint to the left of $P$.
Thus, each color used in the final phase belongs neither to $C_{init}$ nor to $C_{sep}$.
In the final phase, Algorithm uses at least $\ceil{\frac{m}{2}}\frac{2d}{\log{d}+3}$ colors.

In total, Algorithm uses at least $\brac{2\floor{\frac{m}{2}} + \floor{\frac{m}{2}} + 2\ceil{\frac{m}{2}}}\frac{d}{\log{d}+3} = \floor{\frac{5m}{2}}\frac{d}{\log{d}+3}$
different colors.
On the other hand, the presented set of intervals can be easily colored using $\brac{\floor{\frac{m}{2}}+\ceil{\frac{m}{2}}}\brac{\floor{\frac{d}{k}} + \log{d} + 3} = m\brac{\floor{\frac{d}{k}}+\log{d}+3}$ colors.

\end{proof}

Note that for $k = +\infty$, the strategy $HS_{\infty,d}$ becomes independent of $k$-cardinality constraint.
This gives two new bounds on the competitive ratio for on-line $d$-dimensional interval coloring problems.

\begin{theorem}\label{thm:d_weighted_lower_general}
For every $d \geq 2$ and $n \in \mathbb{N}_{+}$,
there is a strategy for Presenter that forces Algorithm
to use at least $(5m-3)\frac{d}{\log{d}+3}$ different colors
in the on-line $d$-dimensional interval coloring game while the constructed set of intervals is $m(\log{d}+3)$-colorable.
\end{theorem}

\begin{theorem}\label{thm:d_weighted_lower_unit}
For every $d \geq 2$ and $n \in \mathbb{N}_{+}$,
there is a strategy for Presenter that forces Algorithm
to use at least $\floor{\frac{5m}{2}}\frac{d}{\log{d}+3}$ different colors
in the on-line $d$-dimensional unit interval coloring game while the constructed set of intervals is $m(\log{d}+3)$-colorable.
\end{theorem}


\begin{thebibliography}{1}

\bibitem{AzarCIKS13}
Yossi Azar, Ilan~Reuven Cohen, Seny Kamara, and Bruce Shepherd.
\newblock Tight bounds for online vector bin packing.
\newblock In {\em {STOC} 2013: 45th Annual {ACM} Symposium on Theory of
  Computing, Palo Alto, {CA}, {USA}, June 2013. Proceedings}, pages 961--970,
  2013.

\bibitem{EpsteinL05_2}
Leah Epstein and Meital Levy.
\newblock Online interval coloring and variants.
\newblock In {\em {ICALP} 2005: 32nd International Colloquim on Automata,
  Languages and Programming, Lisbon, Portugal, July 2005. Proceedings}, volume
  3580 of {\em Lecture Notes in Computer Science}, pages 602--613, 2005.

\bibitem{EpsteinL05}
Leah Epstein and Meital Levy.
\newblock Online interval coloring with packing constraints.
\newblock In {\em {MFCS} 2005: 30th International Symposium on Mathematical
  Foundations of Computer Science, Gda\'{n}sk, Poland, August 2005.
  Proceedings}, volume 3618 of {\em Lecture Notes in Computer Science}, pages
  295--307, 2005.

\bibitem{HalldorssonS94}
Magn{\'{u}}s~M. Halld{\'{o}}rsson and Mario Szegedy.
\newblock Lower bounds for on-line graph coloring.
\newblock {\em Theoretical Computer Science}, 130(1):163 -- 174, 1994.

\bibitem{KiersteadT81}
Henry~A. Kierstead and William~T. Trotter.
\newblock An extremal problem in recursive combinatorics.
\newblock In {\em 12th Southeastern Conference on Combinatorics, Graph Theory
  and Computing, Baton Rouge, LA, USA, March 1981. Proceedings, vol. {II}},
  volume~33 of {\em Congressus Numerantium}, pages 143--153, 1981.

\bibitem{LovaszST89}
Laszlo Lovasz, Michael Saks, and W.T. Trotter.
\newblock An on-line graph coloring algorithm with sublinear performance ratio.
\newblock {\em Discrete Mathematics}, 75(1):319 -- 325, 1989.

\end{thebibliography}

\end{document}